%% file: Opt_L1_Heston_4_arxiv.tex
\documentclass[10pt, a4paper]{article}
\usepackage{graphicx,latexsym, amsmath,amsfonts}
\usepackage{amsthm}
\usepackage{enumerate}
\usepackage{bbm}
\usepackage{subfig}

\usepackage{epsfig}
\usepackage{color}
\usepackage[nocompress]{cite}

\begin{document}

	%\linespread{1.3}
	%the previous command produces a document with "line and a half spacing"
	
	\renewcommand{\d}{d}
	
	\newcommand{\E}{\mathbb{E}}
	\newcommand{\PP}{\mathbb{P}}
	\newcommand{\DD}{\mathbb{D}}
	\newcommand{\R}{\mathbb{R}}
	\newcommand{\cD}{\mathcal{D}}
	\newcommand{\cF}{\mathcal{F}}
	\newcommand{\cK}{\mathcal{K}}
	\newcommand{\N}{\mathbb{N}}
	\newcommand{\fracs}[2]{{ \textstyle \frac{#1}{#2} }}
	\newcommand{\sign}{\operatorname{sign}}
	
	\newtheorem{theorem}{Theorem}[section]
	\newtheorem{lemma}[theorem]{Lemma}
	\newtheorem{coro}[theorem]{Corollary}
	\newtheorem{defn}[theorem]{Definition}
	\newtheorem{assp}[theorem]{Assumption}
	\newtheorem{expl}[theorem]{Example}
	\newtheorem{prop}[theorem]{Proposition}
	\newtheorem{proposition}[theorem]{Proposition}
	\newtheorem{corollary}[theorem]{Corollary}
	\newtheorem{remark}[theorem]{Remark}
	\newtheorem{notation}[theorem]{Notation}

	\def\a{\alpha} \def\g{\gamma}
	\def\e{\varepsilon} \def\z{\zeta} \def\y{\eta} \def\o{\theta}
	\def\vo{\vartheta} \def\k{\kappa} \def\l{\lambda} \def\m{\mu} \def\n{\nu}
	\def\x{\xi}  \def\r{\rho} \def\s{\sigma}
	\def\p{\phi} \def\f{\varphi}   \def\w{\omega}
	\def\q{\surd} \def\i{\bot} \def\h{\forall} \def\j{\emptyset}
	
	\def\be{\beta} \def\de{\delta} \def\up{\upsilon} \def\eq{\equiv}
	\def\ve{\vee} \def\we{\wedge}
	
	\def\t{\tau}
	
	\def\F{{\cal F}}
	\def\T{\tau} \def\G{\Gamma}  \def\D{\Delta} \def\O{\Theta} \def\L{\Lambda}
	\def\X{\Xi} \def\S{\Sigma} \def\W{\Omega}
	\def\M{\partial} \def\N{\nabla} \def\Ex{\exists} \def\K{\times}
	\def\V{\bigvee} \def\U{\bigwedge}
	
	\def\1{\oslash} \def\2{\oplus} \def\3{\otimes} \def\4{\ominus}
	\def\5{\circ} \def\6{\odot} \def\7{\backslash} \def\8{\infty}
	\def\9{\bigcap} \def\0{\bigcup} \def\+{\pm} \def\-{\mp}
	\def\<{\langle} \def\>{\rangle}
	
	\def\lev{\left\vert} \def\rev{\right\vert}
	\def\1{\mathbf{1}}

	\newcommand\wD{\widehat{\D}}
	\newcommand\EE{\mathbb{E}}
	
	\newcommand{\ls}[1]{\textcolor{red}{\tt Lukas: #1 }}

	\title{ \bf The  order barrier for the $L^1$-approximation of the log-Heston SDE at a single point}
	
	\author{Annalena Mickel \footnote{Mathematical Institute and   DFG Research Training Group 1953, 
			University of Mannheim,  B6, 26, D-68131 Mannheim, Germany,  \texttt{amickel@mail.uni-mannheim.de}}
		\and   Andreas Neuenkirch  \footnote{Mathematical Institute, 
			University of Mannheim,  B6, 26, D-68131 Mannheim, Germany, \texttt{aneuenki@mail.uni-mannheim.de}} 
	}

	\date{\today}
	
	\maketitle
	
	\begin{abstract} 
		We study the $L^1$-approximation of the log-Heston SDE at the terminal time point by arbitrary methods that use an equidistant discretization of the driving Brownian motion. We show that such methods can achieve at most order $  \min \{ \nu, \tfrac{1}{2} \}$, where $\nu$ is the Feller index of the underlying CIR process. As a consequence Euler-type schemes are optimal for  $\nu \geq 1$, since  they have convergence order $\tfrac{1}{2}-\varepsilon$ for $\varepsilon >0$ arbitrarily small in this regime.
		\medskip
		
		\noindent \textsf{{\bf Key words: } \em
			Heston model, CIR process, $L^1$-error, lower error bounds, optimal approximation, Euler-type methods
			\\
		}
		
	\end{abstract}

\section{Introduction and main results}
The log-Heston model is a popular stochastic volatility model in mathematical finance.
It is given by the two-dimensional stochastic differential equation (SDE)
\begin{equation}  
	\begin{aligned}
		dX_t &= \big(\mu-\tfrac12 V_t\big) dt + \sqrt{V_t} \Big(\rho d W_t + \sqrt{1-\rho^2} d B_t\Big), \\
		dV_t &= \kappa (\theta-V_t) dt + \sigma \sqrt{V_t} dW_t, \label{MiNe:hes-log} 
	\end{aligned}\qquad t\in[0,T].
\end{equation}
Here $W=(W_t)_{t \in [0,T]}$,  $B=(B_t)_{t \in [0,T]}$ are two independent Brownian motions,  $\rho \in [-1,1]$,
$\mu \in \mathbb{R}$, $\kappa, \theta, \sigma >0$ and the initial values are assumed to be deterministic, i.e., $X_0=x_0$ and $V_0=v_0$ with $x_0 \in \mathbb{R}$ and $v_0>0$.

The stochastic volatility process $(V_t)_{t \in [0,T]}$ of the log-asset price process $(X_t)_{t \in [0,T]}$ is the so-called CIR process, which takes positive values only. This process goes back to the works of Feller in the 1950s, see \cite{MiNe:Feller}. The quantity
$$ \nu = \frac{2 \kappa \theta}{\sigma^2}$$
is called Feller index. Since the square-root function, which appears in the coefficients of SDE \eqref{MiNe:hes-log}, is not globally Lipschitz continuous, SDE \eqref{MiNe:hes-log} does not satisfy the so-called standard assumptions. Much effort has been devoted to the numerical analysis of SDEs under non-standard assumptions in the last 25 years, in particular to the CIR process and the Heston SDE, see, e.g.,  \cite{MiNe:JentzenHefter,MiNe:Mario, MiNe:MiNe} for a comprehensive survey on the strong approximation of the latter SDEs.

\smallskip

In this note, we focus on the optimal $L^1$-approximation of the full log-Heston SDE \eqref{MiNe:hes-log} at the terminal time point $T$ given an equidistant discretization of the driving Brownian motion, that is on
\begin{displaymath}
	e(N)= \inf_{u \in \mathcal{U}(N)}	\, \mathbb{E} \left[ \left\| u( W_{t_1},W_{t_2}, \ldots, W_{t_N}, B_{t_1},B_{t_2}, \ldots, B_{t_N})-\left(
	\begin{array}{c}
		X_T\\
		V_T\\
	\end{array}
	\right) \right \|_1\right],
\end{displaymath}
where $\mathcal{U}(N)$ is the set of measurable functions  $u \colon \mathbb{R}^{2N} \rightarrow \mathbb{R}^2$  and  $t_k=k \frac{T}{N}$ for all $k\in\{0,...,N\}$.

\smallskip

The optimal approximation of SDEs under standard assumptions, that is for SDEs with globally Lipschitz and sufficiently smooth coefficients, has been studied intensively since the pioneering work of Clark and Cameron \cite{MiNe:cam_clark} from 1980;  see Remark \ref{MiNe:cc-tmg}. Recently the analysis of the optimal approximation of SDEs has been extended to the case of non-standard coefficients, see \cite{MiNe:JMGY,MiNe:Yaros1,MiNe:Yaros2} and \cite{MiNe:HeHe,MiNe:HeHeMG,MiNe:JentzenHefter}.
While  the first group of articles constructs SDEs with arbitrary slow best possible convergence rates, the second group establishes results, which apply to the optimal approximation of the  CIR process. In particular, the works \cite{MiNe:HeHeMG} and \cite{MiNe:JentzenHefter} yield that
\begin{displaymath} 
	\liminf_{N \rightarrow \infty} \, N^{\min \{\nu,1 \}} \, \inf_{g \in \mathcal{G}(N)}	\, \mathbb{E}\left[  \big| g( W_{t_1},W_{t_2}, \ldots, W_{t_N}) -V_T \big|\right]  >0,
\end{displaymath}
where $\mathcal{G}(N)$ is the set of measurable functions $ g \colon \mathbb{R}^N \rightarrow \mathbb{R}$  and  $t_k=k \frac{T}{N}$ for  all $k\in\{0,...,N\}$.

\smallskip

So, which rates are  best possible for the approximation of the log-Heston SDE? If $|\rho| \neq 1$ SDE \eqref{MiNe:hes-log} has non-commutative noise, i.e., the diffusion coefficients 
$$ \sigma^{(1)}(x,v)= \left( \begin{array}{c} \rho \sqrt{v} \\ \sigma \sqrt{v}
\end{array} \right), \qquad \sigma^{(2)}(x,v) = \left( \begin{array}{c} \sqrt{1-\rho^2} \sqrt{v} \\ 0
\end{array} \right) $$
do not satisfy the commutativity condition
$$ \sum_{i=1}^{2} \sigma_i^{(1)} \partial_i\sigma^{(2)}=  \sum_{i=1}^{2} \sigma_i^{(2)} \partial_i\sigma^{(1)}.$$
Therefore, it is reasonable to  expect an additional cut-off of the best possible convergence order at $\tfrac{1}{2}$, as in the case of standard coefficients. This article is dedicated to 
confirm this expectation and to this end we 
establish the following new result:

\begin{theorem}  \label{MiNe:thm:low-bound} Let $|\rho| \neq 1$, $t_k=k\Delta t$ for all $k\in\{0,...,N\}$ with $\Delta t=\frac{T}{N}$, let $\mathcal{U}(N)$ be the set of measurable functions  $u \colon \mathbb{R}^{2N} \rightarrow \mathbb{R}^2$
	and
	$$	e(N)= \inf_{u \in \mathcal{U}(N)}	\, \mathbb{E} \left[ \left\| u( W_{t_1},W_{t_2}, \ldots, W_{t_N}, B_{t_1},B_{t_2}, \ldots, B_{t_N})-\left(
	\begin{array}{c}
		X_T\\
		V_T\\
	\end{array}
	\right) \right \|_1\right] 	.
	$$	
	Then we have   that
	\begin{displaymath}
		\liminf_{N \rightarrow \infty} \,\,  N^{\min\{\nu,\frac{1}{2}\}}  \, e(N) >0.
	\end{displaymath}
\end{theorem}
The so far best upper bound  for $e(N)$   has been established in \cite{MiNe:MiNe}, where suitable Euler-type discretizations as, e.g., $	\hat{x}_{0}=x_0, 	\hat{v}_{0}=v_0 $ and 
\begin{equation}  \label{MiNe:Euler-Hes}\begin{aligned}
		\hat{x}_{t_{k+1}}  =	\hat{x}_{t_{k}}   & + \big( \mu - \tfrac{1}{2} 	\hat{v}^+_{t_{k}} \big)(t_{k+1}-t_k) \\ &  +
		\sqrt{\hat{v}^+_{t_{k}}} \Big(\rho \big(W_{t_{k+1}}-W_{t_{k}}\big)  + \sqrt{1-\rho^2} \big(B_{t_{k+1}}-B_{t_{k}}\big)\Big),  \\	\hat{v}_{t_{k+1}} = \hat{v}_{t_{k}} &+ \kappa\big(\theta-\hat{v}^+_{t_{k}}\big)(t_{k+1}-t_k)+\sigma\sqrt{\hat{v}^+_{t_k}}\big(W_{t_{k+1}}-W_{t_{k}}\big),
	\end{aligned}
\end{equation}
for $k=0, \ldots, N-1$, have been analyzed.

\begin{proposition} Let  $t_k=k\Delta t$ for all $k\in\{0,...,N\}$ with $\Delta t=\frac{T}{N}$, let $(\hat{x}_N,\hat{v}_N)$ be given by \eqref{MiNe:Euler-Hes}	
	and let $\varepsilon >0$. Then we have 
	$$ 	\limsup_{N \rightarrow \infty} \, \, N^{\min \{ \frac{\nu}{2}, \frac{1}{2}\}- \varepsilon } \, e(N) \, \leq \, 
	\,  \limsup_{N \rightarrow \infty}\, \,  N^{\min \{ \frac{\nu}{2}, \frac{1}{2}\}- \varepsilon } \, \mathbb{E} \left[ \left\|  \left( \begin{array}{c}
		\hat{x}_N\\
		\hat{v}_N
	\end{array} \right) -\left(
	\begin{array}{c}
		X_T\\
		V_T\\
	\end{array}
	\right) \right \|_1\right] 	\,  = \, 0.
	$$	
\end{proposition}

Thus for $\nu \geq 1$,  Euler-type schemes as \eqref{MiNe:Euler-Hes} obtain the optimal convergence order $\tfrac{1}{2}$, up to an arbitrarily small $\varepsilon >0$. We strongly suppose that the lower error bound in Theorem 1.1 is sharp, i.e., that we  have 
$$ 	\limsup_{N \rightarrow \infty} \,\,  N^{ \min \{\nu, \frac{1}{2}\}} \, e(N) < \infty. $$

The remainder of this note is structured as follows: In the next subsection we collect some remarks, while
Subsection \ref{MiNe:subsec:notation} summarizes some notations.
Section \ref{MiNe:sec:prem} contains then some preliminary results, while the proof of Theorem \ref{MiNe:thm:low-bound} is given in Section  \ref{MiNe:sec:proof}.

\subsection{Remarks}

\begin{remark}
	
	For the optimal approximation of the CIR process, i.e., for
	$$
	e_V(N)= \, \inf_{g \in \mathcal{G}(N)}	\, \mathbb{E}\left[  \big| g( W_{t_1},W_{t_2}, \ldots, W_{t_N}) -V_T \big|\right]
	$$ the best known upper bounds are as follows: For all $\varepsilon >0$ we have
	$$ \limsup_{N \rightarrow \infty } \, N^{\alpha(\nu,\varepsilon)} \, e_V(N) < \infty$$
	with
	$$ \alpha(\nu,\varepsilon)= \left \{ \begin{array}{ccl} 1 & \text{for}  & \nu >2, \\ \tfrac{1}{2}  & \text{for} & \nu \in (1,2] \cup   \{\tfrac{1}{2}\}, \\
		\min \{ \nu,\tfrac{1}{2}\} - \varepsilon &  \text{for}  & \nu \in (0,1]\setminus \{\tfrac{1}{2}\}, \end{array} \right. $$
	see \cite{MiNe:DNS,MiNe:AA2,MiNe:Mario,MiNe:HeHe}.
\end{remark}

\begin{remark} \label{MiNe:cc-tmg} 
	The pioneering work on optimal approximation of stochastic differential equations is \cite{MiNe:cam_clark}.
	Clark and Cameron studied in particular the optimal $L^2$-approximation  of 	
	\begin{equation*}  
		\begin{aligned}
			dX_t &=  V_tdB_t, \\
			dV_t &=  dW_t, 
		\end{aligned}\qquad t\in[0,1],
	\end{equation*} at the final time point by
	an equidistant discretization of the driving Brownian motion.
	Here the optimal method is given by
	$$ \mathbb{E} \big[X_1  \big | W_{\frac{1}{N}},...,W_1, B_{\frac{1}{N}},...,B_1 \big]
	$$
	and  one has
	$$  \left( \mathbb{E}\left[ \left| X_1 -\mathbb{E} \big[X_1  \big | W_{\frac{1}{N}},...,W_1, B_{\frac{1}{N}},...,B_1 \big] \right|^2\right] \right)^{1/2}  = \frac{1}{2} N^{-1/2}.
	$$
	
	Since then, for various error criteria a detailed and exhaustive study of the optimal approximation of general SDEs under standard assumptions has been carried out, even for adaptive discretizations, see, e.g., \cite{MiNe:Newton1986,MiNe:Newton1990,MiNe:Newton1991,MiNe:CastellGaines,MiNe:CambanisHu,MiNe:HMGR1,MiNe:HMGR2,MiNe:HMGR3,MiNe:HMGR4,MiNe:tmg-habil,MiNe:tmg2002,MiNe:tmg2004}
	and \cite{MiNe:survey} for a survey.
	In particular,	for the strong approximation under standard assumptions at the terminal time point, the results can be categorized along  the so-called commutativity condition, which we mentioned above for the log-Heston SDE. If the commutativity condition is satisfied, then the best possible convergence order is at least one, and if this condition fails, then the best possible convergence order is one half.
	
	A comparison with Theorem \ref{MiNe:thm:low-bound} illustrates now nicely the effect of non-standard coefficients, namely  a deterioration of the best possible convergence order for small Feller indices.
	
\end{remark}

\begin{remark}
	In Chapter 10 of \cite{MiNe:Mi-Diss} the optimal $L^2$-approximation of the log-asset component of the more general stochastic volatility model
	\begin{equation}  \label{MiNe:gen}
		\begin{aligned}
			dX_t &= \big(\mu-\tfrac{1}{2}f^2\left(V_t\right)\big) dt + f\left(V_t\right) \Big(\rho dW_t + \sqrt{1-\rho^2} dB_t\Big), \\
			dV_t &= b\left(V_t\right) dt + \sigma\left(V_t\right) dW_t,
		\end{aligned}\qquad t\in[0,T],
	\end{equation}
	by arbitrary methods that use  $N$ evaluations of each Brownian motion
	is analyzed. Here $V=(V_t)_{t\in [0,T]}$ takes values in an open set $D \subseteq \mathbb{R}$, $f,b,\sigma \colon D \rightarrow \mathbb{R}$  are appropriate functions, $\rho \in [-1,1]$, $ \mu \in \mathbb{R}$ and $W=(W_t)_{t \in [0,T]}$, $B=(B_t)_{t \in [0,T]}$ are independent Brownian motions.

	In particular, let
	$$ \Pi(N)= \left \{ (s_i,t_i)_{i=1}^N\colon \,   (s_i,t_i) \in [0,T]^2, i=1, \ldots, N, \, s_N=t_N=T \right\}$$
	and
	let $\mathcal{H}(N)$ be the set of all measurable functions from $\mathbb{R}^{2N}$ to $\mathbb{R}$. Then,
	Theorem 10.3 and Proposition 10.12 in \cite{MiNe:Mi-Diss} establish conditions on \eqref{MiNe:gen} such that for
	\begin{displaymath}
		e_{X,2}(N)\, = \, \inf_{ (s_i,t_i)_{i=1}^N  \in \Pi(N)} \, \, \inf_{h \in \mathcal{H}(N)}	\quad \left( \mathbb{E} \left[ \big| h( W_{s_1}, \ldots, W_{s_N}, B_{t_1}, \ldots, B_{t_N})-X_T \big|^2\right]\right)^{1/2}
	\end{displaymath}	
	we have
	$$ c(b,\sigma,f,\rho, T, v_0)\leq  
	\liminf_{N\rightarrow\infty} N^{1/2}    e_{X,2}(N) \leq 	\limsup_{N\rightarrow\infty} N^{1/2}  e_{X,2}(N) \leq  C(b,\sigma,f,\rho, T, v_0)
	$$	
	with
	$$ c(b,\sigma,f,\rho, T, v_0) = \sqrt{ \frac{1-\rho^2}{6}} \int_{0}^{T} \left( \mathbb{E}\left[(f'\sigma)^2(V_{t})\right]\right)^{1/2}dt$$
	and
	$$ C(b,\sigma,f,\rho, T, v_0) = \sqrt{ \frac{1-\rho^2}{4}} \int_{0}^{T} \left( \mathbb{E}\left[(f'\sigma)^2(V_{t})\right]\right)^{1/2}dt . $$ 
	For the log-Heston model, these conditions are fulfilled for $\nu >2$. In this case we also have
	$$ (f'\sigma)(v)= \frac{\sigma}{2}, \qquad v  >0, $$ and the optimal discretization is equidistant.

\end{remark}

\subsection{Notations}\label{MiNe:subsec:notation}
As already mentioned, we will work with an equidistant discretization
$	t_k = k \Delta t$, $k\in \{0, \ldots, N\},$
where $\Delta t= T/ N$ and $N \in \mathbb{N}$. Furthermore, we define $n(t)=\max\{k\in \{0,...,N\}\colon t_k\le t \} $ and $\eta(t)=t_{n(t)}.$ We also use the standard notation $\| \cdot \|_p$ for the $p$-norms on $\mathbb{R}^d$.
Constants whose values depend only on $T, x_0, v_0, \kappa, \theta, \sigma, \mu, \rho$  will be denoted in the following by $C$, regardless of their value. %Other dependencies will be denoted by subscripts, that is $C_{h,p}$ means e.g.  that this constant depends additionally on  a function $h$ and a parameter $p$. 
Moreover, the value of all these constants can change from line to line. The indicator function of a set $F$ will be denoted by $\chi_F$. Finally, we will work on  a  filtered probability space $(\Omega, \mathcal{F}, (\mathcal{F}_t)_{t\in[0,T]}, P)$, where the filtration satisfies the usual conditions, and (in-)equalities between random variables or random processes are understood $P$-a.s.~unless mentioned otherwise.

\section{Preliminary  results}\label{MiNe:sec:prem}
We collect here a couple of preliminary results.
The first one is a well-known result for the CIR process (see, e.g., Section 3 in \cite{MiNe:DNS} and Theorem 3.1 in \cite{MiNe:HK} for (i) and  Lemma 2.2 in \cite{MiNe:MiNe} for (ii)). 
\begin{lemma}\label{MiNe:bounded}
	(i) Let $p>-\nu$. Then we have
	\begin{displaymath}
		\sup\limits_{t\in[0,T]}\mathbb{E}\left[V^{p}_t\right]<\infty.
	\end{displaymath}
	(ii) Let $p \geq 1$. Then we have 
	\begin{displaymath}	\sup_{s,t \in[0,T]}\mathbb{E}\left[ \frac{\left|V_t-V_s\right|^p}{\left|t-s\right|^{p/2}} \right]
		<\infty.
	\end{displaymath}
\end{lemma}

In the case $\nu \geq 1$, the CIR process takes values in $(0,\infty)$ and an application of It\=o's lemma gives for $U_t =\sqrt{V_t}$ the SDE
\begin{align*}
	U_t = \sqrt{v_0} + \int_{0}^{t}\left(\left(\frac{\kappa\theta}{2}-\frac{\sigma^2}{8}\right)\frac{1}{U_s}-\frac{\kappa}{2}U_s\right)ds +\frac{\sigma}{2}W_t, \qquad t \in [0,T].
\end{align*} 
The next lemma extends this result to the case $\nu \in (1/2,1)$, in which the CIR process also hits zero and takes values in $[0, \infty)$.
\begin{lemma}\label{MiNe:lamperti}
	The process $\left(U_t\right)_{t\in[0,T]}$ defined by $U_t =\sqrt{V_t}$, $t \in [0,T]$, is adapted and has continuous sample paths. Furthermore, if  $\nu>1/2$, it satisfies 
	\begin{align*}
		U_t = \sqrt{v_0} + \int_{0}^{t}\left(\left(\frac{\kappa\theta}{2}-\frac{\sigma^2}{8}\right)\frac{1}{U_s}-\frac{\kappa}{2}U_s\right)\chi_{\{U_s\in(0,\infty)\}}ds +\frac{\sigma}{2}W_t
	\end{align*}
	for all $t \in [0,T]$.
	\begin{proof} {} \, 
		This follows from Lemma 3.2 of \cite{MiNe:HuJeNo}. 
	\end{proof}
\end{lemma}

In the following, we set
\begin{equation*}  A_t= \int_{0}^{t}\left(\left(\frac{\kappa\theta}{2}-\frac{\sigma^2}{8}\right)\frac{1}{U_s}-\frac{\kappa}{2}U_s\right)\chi_{\{U_s\in(0,\infty)\}}ds, \qquad t \in [0,T].\end{equation*}

\begin{lemma}\label{MiNe:Lp}  Let $\nu > 1/2$, $p \geq 2$.
	It holds that
	\begin{align*}
		\sup_{s,t \in[0,T]}\mathbb{E}\left[ \frac{\left|\sqrt{V_t}-\sqrt{V_s}\right|^p}{\left|t-s\right|^{p/4}} \right]
		+
		\sup_{s,t \in[0,T]}\mathbb{E}\left[ \frac{\left|A_t-A_s\right|^p}{\left|t-s\right|^{p/4}} \right]<\infty.
	\end{align*}
	\begin{proof} {} \, 
		First note that Lemma  \ref{MiNe:bounded} (ii) and $|\sqrt{V_t}-\sqrt{V_s}| \leq \sqrt{|V_t-V_s|}$ imply that
		\begin{displaymath}	\sup_{s,t \in[0,T]}\mathbb{E}\left[ \frac{\left|\sqrt{V_t}-\sqrt{V_s}\right|^p}{\left|t-s\right|^{p/4}} \right] \leq \sup_{s,t \in[0,T]}\mathbb{E}\left[ \frac{\left|V_t-V_s\right|^{p/2}}{\left|t-s\right|^{p/4}} \right] < \infty.
		\end{displaymath}
		The second estimate follows then from $A_t=\sqrt{V_t}-\sqrt{v_0}-\frac{\sigma}{2} W_t$, $t \in [0,T]$,
		$$ 	\sup_{s,t \in[0,T]} \mathbb{E}\left[ \frac{\left|W_t-W_s\right|^p}{|t-s|^{p/2}}  \right]  < \infty $$
		and the Minkowski inequality.
	\end{proof}
\end{lemma}

\medskip

As already mentioned, the question of the best possible $L^1$-approximation of the CIR process at the final time point has been answered by the works of Hefter, Herzwurm, Jentzen and M\"uller-Gronbach:
\begin{theorem}[Corollary 14 in \cite{MiNe:HeHeMG} and Theorem 1 in \cite{MiNe:JentzenHefter}]
	It holds that
	\begin{displaymath}
		\liminf_{N \rightarrow \infty} \, N^{\min \{\nu,1 \}} \, \inf_{g \in \mathcal{G}(N)}	\, \mathbb{E}\left[  \big| g( W_{t_1},W_{t_2}, \ldots, W_{t_N}) -V_T \big|\right]  >0,
	\end{displaymath}
	where $\mathcal{G}(N)$ is the set of measurable functions $g\colon \mathbb{R}^N \rightarrow \mathbb{R}$.
\end{theorem}
Since $V$ does not depend on $B$, we have
\begin{align*} & \inf_{g \in \mathcal{G}(N)}	\, \mathbb{E}\left[  \big| g( W_{t_1},W_{t_2}, \ldots, W_{t_N}) -V_T \big|\right] \\ & \qquad  \quad  = \inf_{h \in \mathcal{H}(N)}	\, \mathbb{E} \left[ \big| h( W_{t_1},W_{t_2}, \ldots, W_{t_N}, B_{t_1},B_{t_2}, \ldots, B_{t_N})-V_T \big|\right], \end{align*}
where  
$\mathcal{H}(N)$ is the set of measurable functions  $h \colon \mathbb{R}^{2N} \rightarrow \mathbb{R}$. Therefore our main result follows, if we can show that
\begin{align} \label{MiNe:low-bound:X:prelim}
	\liminf_{N \rightarrow \infty} \,\,  N^{1/2}\inf_{h \in \mathcal{H}(N)}	\, \mathbb{E} \left[ \big| h( W_{t_1},W_{t_2}, \ldots, W_{t_N}, B_{t_1},B_{t_2}, \ldots, B_{t_N})-X_T \big|\right] >0
\end{align}
for  $\nu > \frac{1}{2}$ and $|\rho| \neq 1$.

\medskip

Our proof of \eqref{MiNe:low-bound:X:prelim} will rely on a symmetrization argument. For this we will need an auxiliary result for the quantity
\begin{displaymath} 
	\mathcal{I}(B^{\circ},W) = \int_0^T B_t dW_t - \int_0^T \overline{B}_t dW_t . 
\end{displaymath}
Here, $\overline{B}$ denotes the piecewise linear interpolation of $B$ on the grid $t_0,...,t_N$, i.e., $\overline{B}$ is defined as
\begin{displaymath}
	\overline{B}_t=B_{t_k} + \frac{t-t_k}{t_{k+1}-t_k} (B_{t_{k+1}} -B_{t_k}), \qquad  t \in [t_k,t_{k+1}], \, \,k \in \{ 0, \ldots ,N-1 \},
\end{displaymath}
and we also define
$$ B_t^{\circ}=B_t-\overline{B}_t, \qquad t \in [0,T].$$

\begin{lemma} \label{MiNe:lem-BBdW}
	(i)  Let $n \in \mathbb{N}$,
	\begin{displaymath}
		\tau_{\ell,n}=\frac{\ell}{2^n} \frac{T}{N}, \qquad   \ell \in \{ 0, \ldots, 2^n \}, 
	\end{displaymath}
	and
	\begin{displaymath}
		\mathcal{I}^n(B^{\circ},W) = \sum_{k=0}^{N-1} \sum_{\ell=0}^{2^n-1} B^{\circ}_{t_k+ \tau_{\ell,n} } \left (W_{  t_k + \tau_{\ell+1,n}} - W_{ t_k+ \tau_{\ell,n}} \right ) .
	\end{displaymath} 
	We have that
	\begin{displaymath}
		\mathcal{I}(B^{\circ},W) = \lim_{n \rightarrow \infty}   \mathcal{I}^n(B^{\circ},W)
	\end{displaymath}
	almost surely and in $L^2$. \\	
	(ii)  It holds
	\begin{displaymath}
		\mathcal{I}(B^{\circ},W) \stackrel{d}{=} W_1 \left(\int_0^T |B_t^{\circ}|^2 dt \right)^{1/2}.
	\end{displaymath}
\end{lemma}
\begin{proof}
	
	% We have
	%	\begin{displaymath}
		%		\int_{t_k}^{t_{k+1}} B_t dW_t - \int_{t_k}^{t_{k+1}} \overline{B}_t dW_t = I_1^k-I_2^k \end{displaymath}
	%	with
	%	\begin{displaymath} I_1^k= \int_{t_k}^{t_{k+1}} (B_t-B_{t_k}) dW_t, \qquad I_2^k=  \frac{B_{t_{k+1}} -B_{t_k}}{t_{k+1}-t_k}  \int_{t_k}^{t_{k+1}} (t-t_k) dW_t 
		%	\end{displaymath}
	%	and
	%	\begin{displaymath}
		%		\sum_{\ell=0}^{2^n-1} B^{\circ}_{t_k+ \tau_{\ell,n} } \left (W_{  t_k + \tau_{\ell+1,n}} - W_{ t_k+ \tau_{\ell,n}} \right ) = I_1^{k,n}-I_2^{k,n}
		%	\end{displaymath}
	%	with
	%	\begin{displaymath}
		%		\begin{aligned} 
			%			I_1^{k,n}&=		\sum_{\ell=0}^{2^n-1} (B_{t_k+ \tau_{\ell,n}} -B_{t_k}) \left (W_{  t_k + \tau_{\ell+1,n}} - W_{ t_k+ \tau_{\ell,n}} \right ), \\ I_2^{k,n}&=   \frac{B_{t_{k+1}} -B_{t_k}}{t_{k+1}-t_k}  	\sum_{\ell=0}^{2^n-1} 
			%			\tau_{\ell,n} \left (W_{  t_k + \tau_{\ell+1,n}} - W_{ t_k+ \tau_{\ell,n}} \right ).
			%		\end{aligned}
		%	\end{displaymath}
	(i) 	By straightforward calculations using $B^{\circ}=B-\overline{B}$, the independence of $B$ and $W$ and the It\=o isometry we have
	%	\begin{displaymath}
		%		\mathbb{E}\left[ |I_1^k-I_1^{k,n}|^2\right]  = \frac{1}{2} \left( \frac{T}{N} \right)^2 2^{-n}
		%	\end{displaymath}
	%	and
	%	\begin{displaymath}
		%		\mathbb{E}\left[ |I_2^k-I_2^{k,n}|^2\right]  = \frac{1}{3} \left( \frac{T}{N} \right)^3 2^{-2n}.
		%	\end{displaymath}
	%	Thus 
	\begin{displaymath}
		\mathbb{E} \left[ \left| \mathcal{I}(B^{\circ},W) - \mathcal{I}^n(B^{\circ},W)  \right|^2\right]  \leq C  \cdot N^{-1 } \cdot 2^{-n},
	\end{displaymath}
	which yields the claimed $L^2$-convergence for $n \rightarrow \infty$, and
	also implies
	\begin{displaymath}
		\sum_{n=1}^{\infty} 
		\mathbb{E}\left[  \left| \mathcal{I}(B^{\circ},W) - \mathcal{I}^n(B^{\circ},W)  \right|\right]  < \infty,
	\end{displaymath}
	from which the almost sure convergence follows by an application of the Borel-Cantelli lemma.
	
	(ii)  Recall that  $W$ is independent of $B$ and therefore also of $B^{\circ}$. The conditional law of $\mathcal{I}^n(B^{\circ},W)$ given 
	\begin{displaymath}
		B^{\circ}_{t_k+ \tau_{\ell,n}}=x_{k,\ell}, \quad   \ell \in \{0, \ldots, 2^n-1\}, k \in \{ 0, \ldots, N-1 \} ,
	\end{displaymath}
	is therefore Gaussian with zero mean and variance $  \sum_{k=0}^{N-1} \sum_{\ell=0}^{2^n-1} |x_{k, \ell}|^2 (\tau_{\ell+1,n} -\tau_{\ell,n}) $. We thus have
	\begin{displaymath}
		\mathcal{I}^n(B^{\circ},W) \stackrel{d}{=} W_1  \left( \sum_{k=0}^{N-1} \sum_{\ell=0}^{2^n-1} |B^{\circ}_{t_k+ \tau_{\ell,n}}|^2  \left(  \tau_{\ell+1,n} -\tau_{\ell,n}\right) \right)^{1/2}.
	\end{displaymath}
	Since also
	\begin{displaymath}
		\int_0^T |B^{\circ}_{t}|^2dt = \lim_{n \rightarrow \infty} \sum_{k=0}^{N-1} \sum_{\ell=0}^{2^n-1} |B^{\circ}_{t_k+ \tau_{\ell,n}}|^2  \left(  \tau_{\ell+1,n} -\tau_{\ell,n}\right)
	\end{displaymath}
	almost surely (by continuity of almost all sample paths of $B^{\circ}$), the assertion follows now from part (i).
\end{proof}

\section{Proof of Theorem \ref{MiNe:thm:low-bound}}\label{MiNe:sec:proof}

As already pointed out, for the proof of Theorem \ref{MiNe:thm:low-bound} it remains to show the following result:

\begin{theorem} \label{MiNe:thm:low-bound2}
	Let $\nu > \frac{1}{2}$, $|\rho| \neq 1$ and
	$\mathcal{H}(N)$ be the set of measurable functions  $h \colon \mathbb{R}^{2N} \rightarrow \mathbb{R}$. Then, we have that
	\begin{displaymath}
		\liminf_{N \rightarrow \infty} \, \,  N^{1/2}\inf_{h \in \mathcal{H}(N)}	\, \mathbb{E} \left[ \big| h( W_{t_1},W_{t_2}, \ldots, W_{t_N}, B_{t_1},B_{t_2}, \ldots, B_{t_N})-X_T \big|\right] \ge  c
	\end{displaymath}
	with	\begin{displaymath}
		c= \frac{\sigma T}{8}  \sqrt{1-\rho^2} .
	\end{displaymath}
\end{theorem}
We will simplify the 
analysis of 
\begin{displaymath}
	e_X(N)=	 \inf_{h \in \mathcal{H}(N)}	\, \mathbb{E}\left[  \big| h( W_{t_1},W_{t_2}, \ldots, W_{t_N}, B_{t_1},B_{t_2}, \ldots, B_{t_N}) -X_T \big| \right]  
\end{displaymath}
in several steps until we end up with the optimal $L^1$-approximation of
$\int_0^T B_t dW_t$ by arbitrary methods, which use an equidistant discretization of $B$ and have complete information of $W$, i.e., with the analysis of the quantity
\begin{displaymath}
	\inf_{v \in \mathcal{V}(N)}	\, \mathbb{E}\left[  \left| v( W, B_{t_1},B_{t_2}, \ldots, B_{t_N}) -\int_0^T B_t dW_t \right| \right]   ,
\end{displaymath}
where $\mathcal{V}(N)$ is the set of measurable functions  $v \colon C([0,T]; \mathbb{R}) \times \mathbb{R}^{N} \rightarrow \mathbb{R}$. This quantity can then  be   analyzed in a final step by a symmetrization argument. The latter is a simplified version of Lemma 1 in  \cite{MiNe:JMGY} and is a particular case of the {\it radius of information} concept  in information based complexity, see \cite{MiNe:IBC}.

\subsection{Allowing complete information on $W$}\label{MiNe:complete_info}
Let
\begin{align*} \sigma(\mathcal{H}_N)& = \sigma \left(  W_{t_1},W_{t_2}, \ldots, W_{t_N}, B_{t_1},B_{t_2}, \ldots, B_{t_N} \right),  \\
	\sigma(\mathcal{V}_N)& = \sigma \left(  W, B_{t_1},B_{t_2}, \ldots, B_{t_N} \right)
\end{align*}
and
\begin{align*} \mathcal{Z_H}&=
	\{ Z \colon \Omega \rightarrow \mathbb{R}: \, Z \text{ is } \sigma(\mathcal{H}_N) \text{-measurable}\},  \\
	\mathcal{Z_V}&= \{ Z\colon \Omega \rightarrow \mathbb{R}: \, Z \text{ is } \sigma(\mathcal{V}_N)  \text{-measurable}\}.
\end{align*}
Since
\begin{displaymath}
	\mathcal{Z_H} \subset \mathcal{Z_V}
\end{displaymath} 
it follows that
\begin{displaymath}
	\begin{aligned} 
		&  \inf_{Z \in \mathcal{Z_H}}	\, \mathbb{E}\left[ \big|Z -X_T \big|\right]  =  \inf_{h \in \mathcal{H}(N)}	\, \mathbb{E}\left[  \big| h( W_{t_1},W_{t_2}, \ldots, W_{t_N}, B_{t_1},B_{t_2}, \ldots, B_{t_N}) -X_T \big| \right]    \\  & \quad    \geq  \inf_{Z \in \mathcal{Z_V}}	\, \mathbb{E}\left[  \big| Z -X_T \big|\right]  = \inf_{v \in \mathcal{V}(N)}	\, \mathbb{E}\left[  \big| v( W, B_{t_1},B_{t_2}, \ldots, B_{t_N}) -X_T \big|\right] ,
	\end{aligned}
\end{displaymath}
where $\mathcal{H}(N)$ and $\mathcal{V}(N)$ are as above.
Thus, it is sufficient to analyze the quantity
\begin{equation}  \label{MiNe:trick-1}
	\inf_{Z \in \mathcal{Z_V}}	\, \mathbb{E}\left[  \big| Z -X_T \big|\right]  = \inf_{v \in \mathcal{V}(N)}	\, \mathbb{E} \left[ \big| v( W, B_{t_1},B_{t_2}, \ldots, B_{t_N}) -X_T \big| \right] 
\end{equation}
to obtain a lower bound for $e_X(N)$.

\subsection{Rewriting $X_T$ and removing the measurable part}
Now, we rewrite $X_T$. Note that the CIR process $V=(V_t)_{t \in [0,T]}$ is $\sigma(W)$-measurable (and therefore $\sigma(\mathcal{V}_N)$-measurable) as the unique strong solution of the SDE 
\begin{align*}  	dV_t = \kappa (\theta-V_t) dt + \sigma \sqrt{V_t} dW_t, \quad t \in [0,T], \qquad V_0=v_0>0. \end{align*}

\begin{lemma} For $\nu > \frac{1}{2}$ we have that
	\begin{displaymath}
		X_T= Y_T + \sqrt{1-\rho^2}  \int_0^T A_t d B_t -  \frac{\sigma}{2} \sqrt{1-\rho^2} \int_0^T B_t d W_t ,
	\end{displaymath}
	where
	\begin{displaymath}
		\begin{aligned} Y_T=  x_0& +\frac{\rho}{\sigma}\left(V_T-v_0-\kappa \theta T\right) + \mu T  +\left(\frac{\rho\kappa}{\sigma}-\frac{1}{2}\right)\int_{0}^{T}V_udu \\ &+ \sqrt{1-\rho^2} (\sqrt{V_T}B_T -A_TB_T) \end{aligned}
	\end{displaymath} and
	\begin{equation} \label{MiNe:sqrtV:ohne:W} A_t= \int_{0}^{t}\left(\left(\frac{\kappa\theta}{2}-\frac{\sigma^2}{8}\right)\frac{1}{\sqrt{V_s}}-\frac{\kappa}{2}\sqrt{V_s}\right)\chi_{\{V_s\in(0,\infty)\}}ds, \qquad t \in [0,T].\end{equation}
	In particular,  $A=(A_t)_{t \in [0,T]}$ and $Y_T$ are $\sigma(\mathcal{V}_N)$-measurable.
\end{lemma}

\begin{proof}
	Since
	\begin{displaymath}
		V_T= v_0 + \int_{0}^{T}\kappa(\theta-V_u)du +\sigma\int_{0}^{T} \sqrt{V_u}dW_u
	\end{displaymath}
	we have that
	\begin{equation} \label{MiNe:lem-repr1}  X_T= Y_T^{(1)} +\sqrt{1-\rho^2}\int_{0}^{T}\sqrt{V_u}dB_u
	\end{equation}
	with
	\begin{displaymath}
		Y_T^{(1)}=  x_0+\frac{\rho}{\sigma}\left(V_T-v_0-\kappa \theta T\right) + \mu T  +\left(\frac{\rho\kappa}{\sigma}-\frac{1}{2}\right)\int_{0}^{T}V_udu. 
	\end{displaymath}
	We can use now Lemma \ref{MiNe:lamperti} to write
	\begin{equation} \label{MiNe:semi-sqrt-v}
		\sqrt{V_t} =   \sqrt{v_0} + A_t  + \frac{\sigma}{2} \, W_t, \qquad t \in [0,T],
	\end{equation}
	where $A_t$ is given by \eqref{MiNe:sqrtV:ohne:W}. So $\sqrt{V}$ is a continuous semi-martingale with representation \eqref{MiNe:semi-sqrt-v}. Integration by parts now gives
	\begin{displaymath}
		\int_{0}^{T}\sqrt{V_u} dB_u  = \sqrt{V_T}B_T - \int_0^T B_t d A_t -  \frac{\sigma}{2}  \int_0^T B_t d W_t
	\end{displaymath}
	and
	\begin{displaymath}
		\int_0^T B_t d A_t =A_T B_T - \int_0^T A_t d B_t,
	\end{displaymath}
	respectively.  %gives
	%\begin{displaymath}
	%	\int_{0}^{T}\sqrt{V_u} dB_u  = \sqrt{V_T}B_T - B_TA_T + \int_0^T A_t d B_t -  \frac{\sigma}{2}  \int_0^T B_t d W_t
	%\end{displaymath}
	Together with  \eqref{MiNe:lem-repr1}  this yields
	\begin{displaymath}
		X_T=Y_T^{(1)} + Y_T^{(2)} + \sqrt{1-\rho^2} \left( \int_0^T A_t d B_t -  \frac{\sigma}{2}  \int_0^T B_t d W_t  \right)
	\end{displaymath}
	with
	\begin{displaymath}
		Y_T^{(2)}=  \sqrt{1-\rho^2} ( \sqrt{V_T}B_T - B_TA_T ), 
	\end{displaymath}
	which finishes the proof.
\end{proof}

As a consequence, we have
\begin{displaymath}
	\begin{aligned}
		\inf_{Z \in \mathcal{Z_V}}	\, \mathbb{E}\left[ \big| Z -X_T \big|\right]  & = \inf_{Z \in \mathcal{Z_V}}	\, \mathbb{E}\left[ \left| Z -Y_T  -	\sqrt{1-\rho^2} \left( \int_0^T A_t d B_t -  \frac{\sigma}{2}  \int_0^T B_t d W_t  \right) \right|\right]  \\
		&	= \inf_{\tilde{Z}  \in \mathcal{Z_V}}	\, \mathbb{E}\left[  \left| \tilde{Z}  -	\sqrt{1-\rho^2} \left( \int_0^T A_t d B_t -  \frac{\sigma}{2}  \int_0^T B_t d W_t  \right) \right|\right] 
	\end{aligned}
\end{displaymath}
and it remains to analyze
\begin{equation}  \label{MiNe:trick-2}
	\inf_{v \in \mathcal{V}(N)}  	\, \mathbb{E}\left[  \left| v( W, B_{t_1},B_{t_2}, \ldots, B_{t_N}) -  \int_0^T A_t d B_t +  \frac{\sigma}{2}  \int_0^T B_t d W_t \right|\right]  . 
\end{equation}

\subsection{Removing the smooth part}
Since $A=(A_t)_{t \in [0,T]}$ is smooth enough,  asymptotically $\int_0^T A_t d B_t$   does not matter  for our approximation problem.
\begin{lemma} Let $\nu >\frac{1}{2}$. Then, there exists a constant $C>0$ such that
	\begin{displaymath}
		\mathbb{E}\left[  \left| \int_0^T A_t dB_t - \sum_{i=0}^{N-1}  A_{t_i} (B_{t_{i+1}}-B_{t_i})\right| \right] \leq C \cdot N^{-5/8}.
	\end{displaymath}
\end{lemma}
\begin{proof}
	We have
	\begin{displaymath}
		A_t = \int_0^t a_u du, \qquad t \in [0,T],
	\end{displaymath}
	with
	\begin{displaymath}
		a_u= \left( \frac{4 \kappa \theta - \sigma^2}{8} \frac{1}{ \sqrt{V_u}} -  \frac{\kappa}{2}\sqrt{V_u} \right)\chi_{\{V_u\in(0,\infty)\}}, \qquad u \in [0,T].
	\end{displaymath}
	Since $\nu> \frac{1}{2}$ we have by Lemma \ref{MiNe:bounded} and Lemma \ref{MiNe:Lp} that
	\begin{align}\label{MiNe:aux_smooth_eq}
		\sup_{t \in[0,T]} \mathbb{E}\left[ |a_t|^{\nu + \frac{1}{2}}\right]  < \infty, \qquad  \sup_{s,t \in[0,T]}  \frac{\mathbb{E}\left[ |A_t-A_s|^q\right]}{|t-s|^{q/4}} < \infty
	\end{align}
	for all $q\geq 1$.
	The It\=o isometry now gives
	\begin{displaymath}
		\begin{aligned}
			\mathbb{E} \left[ \left| \int_0^T A_t dB_t - \sum_{i=0}^{N-1}  A_{t_i} (B_{t_{i+1}}-B_{t_i})\right|^2\right]   &= \mathbb{E}\left[  \left| \int_0^T  (A_t-A_{\eta(t)}) dB_t \right|^2\right]  \\   \qquad = \int_0^T \mathbb{E}\left[  \left| A_t-A_{\eta(t)}  \right|^2\right]  dt
			& = \int_0^T \mathbb{E}\left[  \left| A_t-A_{\eta(t)}  \right|  \left|  \int_{\eta(t)}^t a_u du \right| \right]  dt.
		\end{aligned}
	\end{displaymath}
	H\"older's inequality with $p= \nu+ \frac{1}{2}$, $q = \frac{\nu + \frac{1}{2}}{\nu-\frac{1}{2}}$ yields
	\begin{align*}
		\mathbb{E}\left[  \left| A_t-A_{\eta(t)}  \right|  \left|  \int_{\eta(t)}^t a_u du \right| \right]\leq   \left( \mathbb{E} \left[\left| A_t-A_{\eta(t)}  \right|^q\right] \right)^{1/q}  \left( \mathbb{E}  \left[  \left| \int_{\eta(t)}^t a_u du \right|^{p} \right] \right)^{ 1/p}.
	\end{align*}
	Using \eqref{MiNe:aux_smooth_eq} we have that
	\begin{align*}
		\left( \mathbb{E}\left[ \left| A_t-A_{\eta(t)}  \right|^q\right] \right)^{1/q} \leq C \cdot   (\Delta t)^{1/4}.
	\end{align*}
	Moreover, \eqref{MiNe:aux_smooth_eq}  also gives
	\begin{displaymath}
		\mathbb{E}\left[ \left| \int_{\eta(t)}^t a_u du \right|^p\right]  \leq    \left( \sup_{u \in[0,T]} \mathbb{E}\left[ |a_u|^p \right] \right) (t-\eta(t))^p \leq C \cdot  (\Delta t)^p,
	\end{displaymath}
	and so we have
	\begin{displaymath}
		\mathbb{E}\left[  \left| \int_0^T A_t dB_t - \sum_{i=0}^{N-1}  A_{t_i} (B_{t_{i+1}}-B_{t_i})\right|^2 \right]  \leq C \cdot   N^{-5/4}.
	\end{displaymath}
	The assertion follows now from the Lyapunov inequality. 
\end{proof}

Since  $\sum_{i=0}^{N-1}  A_{t_i} (B_{t_{i+1}}-B_{t_i})$ is  $\sigma(\mathcal{V}_N)$-measurable, we obtain that
\begin{displaymath}
	\begin{aligned} 
		& \inf_{v \in \mathcal{V(N)}}  	\, \mathbb{E}\left[ \left| v( W, B_{t_1},B_{t_2}, \ldots, B_{t_N})-
		\int_0^T A_t d B_t +  \frac{\sigma}{2}  \int_0^T B_t d W_t
		\right|\right] 
		\\& \qquad = \inf_{ \tilde{v} \in \mathcal{V}(N)}  	\, \mathbb{E}\left[  \left|  \tilde{v}( W, B_{t_1},B_{t_2}, \ldots, B_{t_N})+ \frac{\sigma}{2}  \int_0^T B_t d W_t  - \int_0^T (A_t-A_{\eta(t)}) d B_t   
		\right|\right] 
		\\ & \qquad  \geq \inf_{ \tilde{v} \in \mathcal{V}(N)}  	\, \mathbb{E}\left[  \left|  \tilde{v}( W, B_{t_1},B_{t_2}, \ldots, B_{t_N})+ \frac{\sigma}{2}  \int_0^T B_t d W_t   
		\right|\right]  - C \cdot   N^{-5/8}
	\end{aligned}
\end{displaymath}
using that $|x|-|y| \leq  |x-y|$ for all $x,y \in \mathbb{R}$.
Consequently, we have reduced our initial problem to the study of
\begin{equation}  \label{MiNe:trick-4}
	\inf_{v \in \mathcal{V}(N)}  	\, \mathbb{E}\left[  \left| v( W, B_{t_1},B_{t_2}, \ldots, B_{t_N})-\int_0^T B_t d W_t
	\right|\right] .
\end{equation}

\subsection{Inserting Brownian bridges and symmetrization}\label{MiNe:BBS}

Our final step relies on a symmetrization argument.
So, let us denote the piecewise linear interpolation of $B$ on the grid $t_0,...,t_N$ by $\overline{B}$, i.e., $\overline{B}$ is defined as
\begin{displaymath}
	\overline{B}_t=B_{t_k} + \frac{t-t_k}{t_{k+1}-t_k} (B_{t_{k+1}} -B_{t_k}), \qquad  t \in [t_k,t_{k+1}], \, \,  k \in \{ 0, \ldots ,N-1 \} .
\end{displaymath}
Then the process $B^{\circ}$ given by
\begin{displaymath}
	B^{\circ}_t=B_t - \overline{B}_t, \qquad t \in [0,T],
\end{displaymath}
is a Brownian bridge on $[t_k,t_{k+1}]$ for all $k  \in \{ 0, \ldots ,N-1\} $, and moreover the processes
\begin{displaymath}
	(B^{\circ}_t)_{t \in [t_0,t_1]}, \, (B^{\circ}_t)_{t \in [t_1,t_2]} \, ,   \, \ldots \, , \,  (B^{\circ}_t)_{t \in [t_{N-1},t_N]}, \,\,  \overline{B}, \,\, W 
\end{displaymath}
are independent. Since
\begin{displaymath}
	\int_0^T \overline{B}_t dW_t = \sum_{k=0}^{N-1} B_{t_k} \left( W_{t_{k+1}}- W_{t_k} \right) +  \sum_{k=0}^{N-1}  \frac{B_{t_{k+1}} -B_{t_k}}{t_{k+1}-t_k}    \int_{t_k}^{t_{k+1}} (t-t_k)  d W_t 
\end{displaymath}
is $\sigma(\mathcal{V}_N)$-measurable, we have that
\begin{equation}
	\begin{aligned}
		& \inf_{v \in \mathcal{V}(N)}  	\, \mathbb{E}\left[  \left| v( W, B_{t_1},B_{t_2}, \ldots, B_{t_N})-\int_0^T B_t d W_t
		\right|\right]  \\ & \qquad \qquad  = 	\inf_{\tilde{v} \in \mathcal{V}(N)}  	\, \mathbb{E}\left[ \left| \tilde{v}( W, B_{t_1},B_{t_2}, \ldots, B_{t_N})- \mathcal{I}(B^{\circ},W) 
		\right|\right] 
	\end{aligned}\nonumber
\end{equation}
with
\begin{displaymath} 
	\mathcal{I}(B^{\circ},W) = \int_0^T B_t dW_t - \int_0^T \overline{B}_t dW_t . 
\end{displaymath}
Furthermore $B^{\circ}$ and $-B^{\circ}$ have the same law, so the independence of $B^{\circ}$ from $(W,\overline{B})$ implies that
\begin{equation*} 
	\left (W, \overline{B},B^{\circ} \right) \stackrel{d}{=} \left (W, \overline{B}, - B^{\circ}  \right).
\end{equation*}
Then, we also have
\begin{displaymath}   \left (W, \overline{B},\mathcal{I}^n(B^{\circ},W) \right) \stackrel{d}{=} \left(W, \overline{B}, - \mathcal{I}^n(B^{\circ},W) \right), 
\end{displaymath} where
\begin{displaymath}
	\mathcal{I}^n(B^{\circ},W) = \sum_{k=0}^{N-1} \sum_{\ell=0}^{2^n-1} B^{\circ}_{t_k+ \tau_{\ell,n} } \left (W_{  t_k + \tau_{\ell+1,n}} - W_{ t_k+ \tau_{\ell,n}} \right ), \qquad  n \in \mathbb{N},
\end{displaymath} 
with
$$ 
\tau_{\ell,n}=\frac{\ell}{2^n} \frac{T}{N}.
$$
Now, Lemma \ref{MiNe:lem-BBdW} implies
\begin{equation*} 
	\left (W, \overline{B},  \mathcal{I}(B^{\circ},W)  \right) \stackrel{d}{=} \left (W, \overline{B}, - \mathcal{I}(B^{\circ},W) \right).
\end{equation*}
Consequently, we have
\begin{displaymath}
	\mathbb{E} \left[ \left| v( W, B_{t_1},B_{t_2}, \ldots, B_{t_N})- \mathcal{I}(B^{\circ},W)
	\right|\right] =  \mathbb{E} \left[ \left| v( W, B_{t_1},B_{t_2}, \ldots, B_{t_N})+ \mathcal{I}(B^{\circ},W)
	\right| \right] 
\end{displaymath}
and so
\begin{displaymath}
	\begin{aligned}
		&  2 \mathbb{E}\left[ \left| \mathcal{I}(B^{\circ},W)	\right|\right]  \\ &  = \mathbb{E} \left[ \left|
		\left( 
		\mathcal{I}(B^{\circ},W) - v( W, B_{t_1}, \ldots, B_{t_N})\right) + \left(  v( W, B_{t_1}, \ldots, B_{t_N}) + 
		\mathcal{I}(B^{\circ},W) \right)
		\right|\right]   \\  & \leq 2 \mathbb{E} \left[ \left| v( W, B_{t_1},B_{t_2}, \ldots, B_{t_N})-\mathcal{I}(B^{\circ},W) \right|\right].
	\end{aligned}
\end{displaymath}
It follows that
\begin{displaymath} 
	\inf_{v \in \mathcal{V}(N)} 	\, \mathbb{E} \left[ \left| v( W, B_{t_1},B_{t_2}, \ldots, B_{t_N})- \mathcal{I}(B^{\circ},W)
	\right| \right] \geq \mathbb{E} \left[ \left| \mathcal{I}(B^{\circ},W)	\right|\right] 
\end{displaymath} and therefore we have
\begin{equation*}
	\begin{aligned}
		\inf_{v \in \mathcal{V}(N)} 	\, \mathbb{E} \left[ \left| v( W, B_{t_1},B_{t_2}, \ldots, B_{t_N})- \mathcal{I}(B^{\circ},W)
		\right|\right]  & \geq    \mathbb{E}\left[ |W_1| \right]  \mathbb{E} \left[\left(\int_0^T   |B_t^{\circ}|^2  dt\right)^{1/2}\right] \\ & \geq   \frac{1}{\sqrt{T}} \mathbb{E}\left[ |W_1| \right] \int_0^T \mathbb{E} \left[ |B_t^{\circ}|\right]  dt
	\end{aligned}
\end{equation*} 
by Lemma \ref{MiNe:lem-BBdW} (ii) and  by Jensen's inequality.
Using
$  \mathbb{E} \left[ |X|\right]  = \sqrt{ \frac{2}{\pi}} \sigma$ for $X  \sim \mathcal{N}(0,\sigma^2)$ we obtain
\begin{displaymath}
	\int_0^T \mathbb{E} \left[ |B_t^{\circ}|\right]  dt = \sqrt{ \frac{2}{\pi} } \int_0^T  \left( \mathbb{E} \left[ |  B_t^{\circ}|^2\right] \right) ^{1/2} dt.
\end{displaymath}
Straightforward calculations give
\begin{displaymath}
	\mathbb{E} \left[ |  B_t^{\circ}|^2\right] = \frac{(t-t_k)(t_{k+1}-t)}{t_{k+1}-t_k}, \qquad t \in [t_k,t_{k+1}],
\end{displaymath}
which in turn yields
\begin{displaymath}
	\int_0^T \mathbb{E} \left[|  B_t^{\circ}|^2 \right]^{1/2}dt    = N  \int_0^{T/N}  \sqrt{\frac{ t(T/N-t)}{T/N}}dt = \sqrt{\frac{T^{3}}{N}} \int_0^1 \sqrt{x(1-x)} dx.
\end{displaymath}
Since $ \int_0^1 \sqrt{x(1-x)} dx = \pi/8$, we have shown that
\begin{equation}\label{MiNe:trick-6}  
	\begin{aligned} N^{1/2}  \inf_{v \in \mathcal{V}(N)}  	& \, \mathbb{E} \left[ \left| v( W, B_{t_1},B_{t_2}, \ldots, B_{t_N})-\int_0^T B_t d W_t
		\right|\right]  \geq \frac{T}{4}. \end{aligned}
\end{equation}

Combining Subsections \ref{MiNe:complete_info}--\ref{MiNe:BBS} with Equations \eqref{MiNe:trick-1}, \eqref{MiNe:trick-2}, \eqref{MiNe:trick-4} and \eqref{MiNe:trick-6} concludes the proof of Theorem \ref{MiNe:thm:low-bound2}.

\bigskip
\bigskip

{\bf Acknowledgments.} \,\, 
%{\it The authors are grateful to the referees for their insightful comments and remarks, which helped to improve this article.}
{\it  Annalena Mickel  is supported by the DFG 
	Research Training Group 1953 "Statistical Modeling of Complex Systems".}

\input{references}

\end{document}

%% file: references.tex
%%%%%%%%%%%%%%%%%%%%%%%% referenc.tex %%%%%%%%%%%%%%%%%%%%%%%%%%%%%%
% sample references
% %
% Use this file as a template for your own input.
%
%%%%%%%%%%%%%%%%%%%%%%%% Springer-Verlag %%%%%%%%%%%%%%%%%%%%%%%%%%
%
% BibTeX users please use
% \bibliographystyle{}
% \bibliography{}
%